\documentclass[a4paper, reqno]{amsart}

\usepackage[utf8]{inputenc}
\usepackage{mathtools}
\usepackage{graphicx}
\usepackage{array}
\usepackage[colorlinks]{hyperref}

\newtheorem{lemma}{Lemma}[section]
\newtheorem{theorem}[lemma]{Theorem}

\newcommand{\R}{\mathbb{R}}

\newcommand{\Pcal}{\mathcal{P}}

\newcommand{\Scal}{\mathcal{S}}

\newcommand{\avgkiss}{\kappa}

\DeclareMathOperator{\trace}{tr}

\DeclareMathOperator{\dens}{dens}
\DeclareMathOperator{\ip}{ip}

\newcommand{\optprob}[1]{{\arraycolsep=0pt%
  \begin{array}{r@{\ }l@{\quad}l}
    #1
  \end{array}}}

\newcommand{\defi}[1]{\textit{#1}}

\newcommand{\tp}{{\sf T}}

\title{Semidefinite programming bounds for the average kissing number}

\author{Maria Dostert}
\address{M. Dostert, \'Ecole polytechnique f\'ed\'erale de Lausanne,
  SB TN, Station 8, 1015 Lausanne, Suisse / Switzerland.}
\email{maria.dostert@epfl.ch}

\author{Alexander Kolpakov}
\address{A. Kolpakov, Institut de math\'ematiques, Rue Emile-Argand 11,
  2000 Neuch\^atel, Suisse / Switzerland.}
\email{kolpakov.alexander@gmail.com}

\author{Fernando Mário de Oliveira Filho}
\address{F.M. de Oliveira Filho, Delft Institute of Applied
  Mathematics, Delft University of Technology, Van Mourik
  Broekmanweg~6, 2628 XE Delft, The Netherlands.}
\email{fmario@gmail.com}

\thanks{The second author was supported by the Swiss National Science
  Foundation project number~PP00P2\_170560.}

\subjclass[2010]{52C17, 90C22, 90C34}

\date{26 March 2020}

\begin{document}

\begin{abstract}
  The average kissing number of~$\R^n$ is the supremum of the average
  degrees of contact graphs of packings of finitely many balls (of any
  radii) in~$\R^n$. We provide an upper bound for the average kissing
  number based on semidefinite programming that improves previous
  bounds in dimensions~$3$,~\dots,~$9$. A~very simple upper bound for
  the average kissing number is twice the kissing number; in
  dimensions~$6$, \dots,~$9$ our new bound is the first to improve on
  this simple upper bound.
\end{abstract}

\maketitle
\markboth{M. Dostert, A. Kolpakov, and F.M. de Oliveira
  Filho}{Semidefinite programming bounds for the average kissing
  number}


\section{Introduction}

A \textit{packing of balls} in~$\R^n$ is a finite set of
interior-disjoint closed balls. The \textit{contact graph} of a
packing~$\Pcal$ is the graph with vertex set~$\Pcal$ in which two
balls~$X$ and~$Y$ are adjacent if they intersect, that is, if they are
tangent to each other.

Contact graphs of packings of disks on the plane are characterized by
the Koebe-Andreev-Thurston theorem~\cite{Stephenson2005}: they are
precisely the (simple) planar graphs. In higher dimensions, no such
simple characterization is known (see the paper by
Glazyrin~\cite{Glazyrin2017} for a nice discussion), and therefore
research has been focused on understanding the behavior of some
specific parameters of contact graphs.

In this paper, we consider the average degree of contact graphs. More
precisely, we are interested in the \textit{average kissing number}
of~$\R^n$, namely
\[
  \avgkiss_n = \sup\{\, \overline{\delta}(G) : \text{$G$ is the
    contact graph of a packing of balls in~$\R^n$}\,\},
\]
where~$\overline{\delta}(G)$ denotes the average degree of~$G$.

Lower bounds for~$\avgkiss_n$ can be obtained by constructions; a
simple idea is to consider lattice packings. Given a
lattice~$\Lambda \subseteq \R^n$ with shortest vectors of length~$d$,
we consider the set of all balls of radius~$d/2$ centered on the
lattice points. These balls have disjoint interiors and so we have a
packing of infinitely many balls. Each ball in this packing has the
same number of tangent balls, called the \textit{kissing number} of
the lattice~$\Lambda$. The \textit{lattice kissing number} of~$\R^n$,
denoted by~$\tau^*_n$, is the largest kissing number of any lattice
in~$\R^n$; immediately we have~$\avgkiss_n \geq \tau^*_n$.  Conway and
Sloane~\cite[Table~1.2]{ConwayS1988} list lower bounds for~$\tau^*_n$,
and hence for~$\avgkiss_n$, for~$n$ up to~128. For~$n = 3$, a
construction of Eppstein, Kuperberg, and Ziegler~\cite{EppsteinKZ2003}
gives~$\avgkiss_3 \geq 12.612$, while~$\tau_3^* = 12$.

On the side of upper bounds, it is easy to see
that~$\avgkiss_n \leq 2 \tau_n$, where~$\tau_n$ is the \textit{kissing
  number} of~$\R^n$, that is, the maximum number of interior-disjoint
unit balls that can simultaneously touch a central unit ball. Indeed,
say~$\Pcal$ is a packing of balls and let~$r(X)$ be the radius of the
ball~$X \in \Pcal$; let~$G = (\Pcal, E)$ be the contact graph
of~$\Pcal$. In~$G$, the number of neighbors of a ball~$X \in \Pcal$
that have radius at least~$r(X)$ is at most the kissing
number~$\tau_n$. So
\[
  |E| \leq \sum_{X \in \Pcal} |\{\,\{X, Y\} \in E : r(X) \leq
  r(Y)\,\}| \leq \tau_n |\Pcal|,
\]
whence the average degree of~$G$ is~$2|E|/|\Pcal| \leq 2 \tau_n$. Though
simple, this bound is still the best known for all~$n \geq 10$.

Kuperberg and Schramm~\cite{KuperbergS1994} gave the first nontrivial
upper bound for the average kissing number in dimension~$3$, proving
that~$\avgkiss_3 \leq 8 + 4 \sqrt{3} =
14.928\ldots$. Glazyrin~\cite{Glazyrin2017} refined their approach and
showed that~$\avgkiss_3 \leq 13.955$; he also extended their result to
higher dimensions and managed to beat the upper bound of~$2\tau_n$
for~$n = 4$ and~$5$. In this paper, we use semidefinite programming to
refine Glazyrin's approach (see~\S\S\ref{sec:glazyrin}
and~\ref{sec:first-sdp} below), obtaining better upper bounds
for~$n = 3$,~\dots,~$9$; see
Table~\ref{tab:bounds}. In~\S\ref{sec:cohn-elkies} we discuss an
alternative approach related to the linear programming bound of Cohn
and Elkies~\cite{CohnE2003} for the sphere packing density.

\begin{table}[tbp]
  \begin{center}
    \begin{tabular}{cccc}
      \textsl{Dimension}&\textsl{Lower bound}
      &\textsl{Previous upper bound}
      &\textsl{New upper bound}\\[3pt]
      \hline\noalign{\vskip3pt}
      3  & 12.612 &   13.955& 13.606\\
      4  & 24     &   34.681& 27.439\\
      5  & 40     &   77.757& 64.022 \\
      6  & 72     &  156 &   121.105\\
      7  & 126    &  268 &   223.144\\
      8  & 240    &  480 &   408.386\\
      9  & 272    &  726 &   722.629\\[3pt]
      \hline
    \end{tabular}
  \end{center}
  \bigskip

  \caption{Lower and upper bounds for the average kissing number. The
    lower bound in dimension~3 was given by Eppstein, Kuperberg, and
    Ziegler~\cite{EppsteinKZ2003}; all other lower are in listed by
    Conway and Sloane~\cite[Table~1.2]{ConwayS1988}. Upper bounds in
    dimensions~$3$, \dots,~$5$ are due to
    Glazyrin~\cite{Glazyrin2017}; all other upper bounds are twice the
    best known upper bound for the kissing number; see Table~1 in
    Machado and Oliveira~\cite{MachadoO2018}.}
  \label{tab:bounds}
\end{table}


\subsection{Notation and preliminaries}

The Euclidean inner product on~$\R^n$ is denoted by
$x \cdot y = x_1 y_1 + \cdots + x_n y_n$ for~$x$, $y \in \R^n$. The
$(n-1)$-dimensional unit sphere is
$S^{n-1} = \{\, x \in \R^n : \|x\| = 1\,\}$; the distance between
points~$x$, $y \in S^{n-1}$ is~$\arccos x \cdot y$. The surface
measure on the $(n-1)$-dimensional sphere of radius~$\rho$ is denoted
by~$\omega_\rho$; we write~$\omega = \omega_1$.

A spherical cap in~$S^{n-1}$ of center~$x \in S^{n-1}$ and
radius~$\alpha$ is the set of all points in~$S^{n-1}$ at distance at
most~$\alpha$ from~$x$, namely
\[
  \{\, y \in S^{n-1} : x \cdot y \geq \cos\alpha\,\}.
\]
The normalized area of this cap is
\[
  \frac{\omega(S^{n-2})}{\omega(S^{n-1})} \int_{\cos\alpha}^1 (1 -
  u^2)^{(n-3)/2}\, du.
\]
Spherical caps are defined similarly for spheres of radius other
than~1. Of course, the normalized area of a cap of radius~$\alpha$ is
the same irrespective of the radius of the sphere, and is given by the
formula above. The area of a spherical cap can be computed by a
recurrence; see Appendix~\ref{ap:area-caps}.

Let~$V$ be a measure space. A kernel is a real-valued 
square-integrable function on~$V \times V$. If~$f\colon V \to \R$ is
square integrable, then~$f\otimes f^*$ is the kernel that
maps~$(x, y)$ to~$f(x) f(y)$.


\section{Glazyrin's upper bound}
\label{sec:glazyrin}

Glazyrin~\cite{Glazyrin2017} refines and extends previous work by
Kuperberg and Schramm~\cite{KuperbergS1994} and obtains as a result
the best upper bounds on the average kissing number in dimensions~$3$,
\dots,~$5$. Here is a short description of Glazyrin's approach,
following his presentation. In~\S\ref{sec:first-sdp} we will see how
Glazyrin's bounds can be improved with the help of semidefinite
programming.

Fix~$\rho > 1$ and the dimension~$n \geq 3$. For~$r > 0$, let~$B_r$ be
a ball of radius~$r$ tangent to the ball of radius~1 centered at the
origin. The intersection of~$B_r$ with the sphere~$\rho S^{n-1}$ of
radius~$\rho$ centered at the origin, if nonempty, is a spherical cap
on~$\rho S^{n-1}$. The normalized area of this spherical cap is
denoted by~$A_{n,\rho}(r)$, that is,
\[
  A_{n,\rho}(r) = \frac{\omega_\rho(B_r \cap \rho
    S^{n-1})}{\omega_\rho(\rho S^{n-1})},
\]
which as a function of~$r$ is monotonically increasing.

\begin{lemma}
  \label{lem:minimum}
  If~$n \geq 3$, $\rho > 1$, and~$r > 0$,
  then~$A_{n,\rho}(r) + A_{n,\rho}(1/r) \geq 2 A_{n,\rho}(1)$.
\end{lemma}

\begin{proof}
If~$\rho \geq 3$, then~$A_{n,\rho}(1) = 0$ and the result follows, so we
assume~$\rho < 3$. For~$s > 0$, let~$B_s$ be a ball of radius~$s$
tangent to the ball of radius~1 centered at the origin. Let us assume
first that both intersections~$B_r \cap \rho S^{n-1}$
and~$B_{1/r} \cap \rho S^{n-1}$ are nonempty and hence that both are
spherical caps in~$\rho S^{n-1}$; let~$\alpha$ and~$\beta$ denote
their radii.

\begin{figure}[tb]
  \begin{center}
    \includegraphics{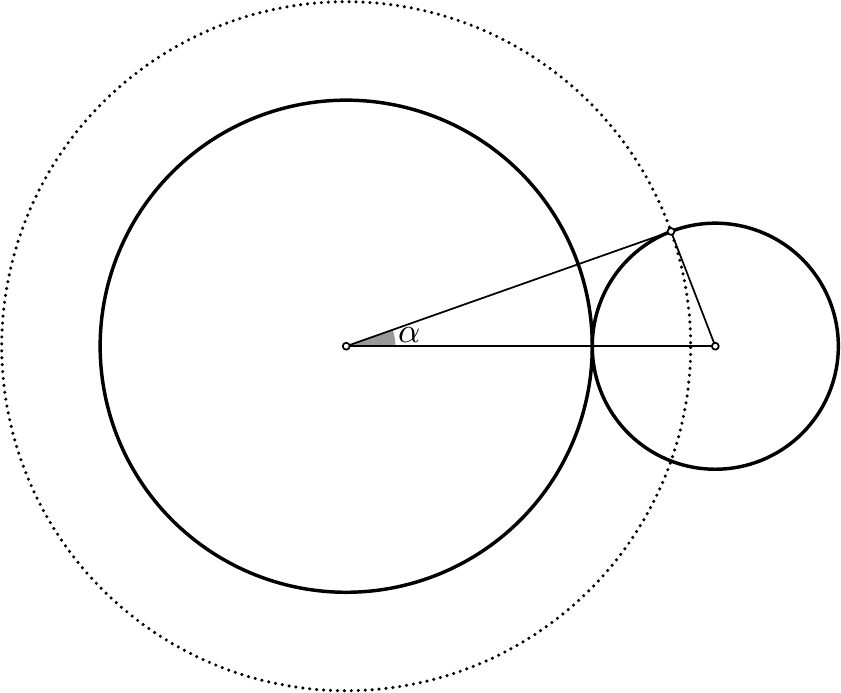}
  \end{center}

  \caption{If the central circle, drawn with full line, has radius~1,
    if the circle drawn with dotted line has radius~$\rho > 1$, and if
    the smaller circle that touches the central circle has radius~$s$,
    then the law of cosines gives
    $s^2 = \rho^2 + (1 + s)^2 - 2\rho(1+s)\cos\alpha$.}
  \label{fig:lem-minimum}
\end{figure}

Using the law of cosines (see Figure~\ref{fig:lem-minimum}) we can
determine both~$\cos\alpha$ and~$\cos\beta$ and as a consequence get
\[
  \cos\alpha + \cos\beta = \frac{\rho^2 + 3}{2\rho}.
\]
Let~$C$ denote the right-hand side above, so~$C \in (1, 2)$
since~$1 < \rho < 3$. Write $x = \cos \alpha$;
then~$\cos\beta = C - x$. Use the formula for the area of a spherical
cap to get
\[
  A_{n,\rho}(r) + A_{n,\rho}(1/r) = K \int_x^1 (1 -
  u^2)^{(n-3)/2}\, du + K \int_{C-x}^1 (1 - u^2)^{(n-3)/2}\, du,
\]
where~$K$ is a positive constant depending only on the
dimension~$n$.

From the above expression, if~$f(x)$ is such
that~$A_{n,\rho}(r) + A_{n,\rho}(1/r) = K f(x)$, then
\[
  f'(x) = -(1 - x^2)^{(n-3)/2} + (1 - (C - x)^2)^{(n-3)/2}.
\]
So~$f$ is monotonically decreasing when~$1 - x^2 \geq 1 - (C-x)^2$,
that is, when $x \in [C-1, C/2]$, and monotonically increasing
when~$x \in [C/2, 1]$. A global minimum of~$f$ is therefore attained
at~$x = C/2$, which implies that~$\alpha = \beta$ and so~$r = 1/r$ and
therefore~$r = 1$, proving the theorem when
both~$B_r \cap \rho S^{n-1}$ and~$B_{1/r} \cap \rho S^{n-1}$ are
nonempty.

Now say~$B_r \cap \rho S^{n-1}$ is empty;
then~$B_{1/r} \cap \rho S^{n-1}$ is not empty. Note~$r < r'$,
where~$r' = (\rho - 1) / 2$; moreover~$B_{r'} \cap \rho S^{n-1}$ is a
single point and so~$A_{n,\rho}(r') = 0$.  Since
also~$B_{1/r'} \cap \rho S^{n-1} \neq \emptyset$, we know
that~$A_{n,\rho}(r') + A_{n,\rho}(1 / r') \geq 2 A_{n,\rho}(1)$ and
hence, since~$A_{n,\rho}$ is monotonically increasing, we get
\[
  A_{n,\rho}(r) + A_{n,\rho}(1/r) \geq A_{n,\rho}(r') + A_{n,\rho}(1/r') \geq
  2 A_{n,\rho}(1),
\]
as we wanted.
\end{proof}

Fix~$\rho > 1$ and consider a unit ball centered at the origin. Any
configuration of pairwise interior-disjoint balls (of any radii)
tangent to the central unit ball covers a certain fraction of the
sphere~$\rho S^{n-1}$ of radius~$\rho$ centered at the origin. The
supremum of this covered fraction taken over all possible
configurations is denoted by~$\dens_n(\rho)$.

\begin{theorem}
  \label{thm:upper-bound}
  If~$n \geq 3$ and~$1 < \rho < 3$, then
  \[
    \avgkiss_n \leq \frac{\dens_n(\rho)}{A_{n,\rho}(1)}.
  \]
\end{theorem}

\begin{proof}
Let~$G = (\Pcal, E)$ be the contact graph of a packing~$\Pcal$ of
balls in~$\R^n$. Denote by~$r(X)$ the radius of a ball~$X \in \Pcal$. On
the one hand, applying Lemma~\ref{lem:minimum} we get
\[
  \sum_{\{X, Y\} \in E} A_{n,\rho}(r(X) / r(Y)) + A_{n,\rho}(r(Y) / r(X))
  \geq 2 A_{n,\rho}(1) |E|.
\]
On the other hand, writing~$N(X)$ for the set of neighbors
of~$X \in \Pcal$, we get
\[
  \begin{split}
  \sum_{\{X, Y\} \in E} A_{n,\rho}(r(X) / r(Y)) + A_{n,\rho}(r(Y) / r(X))
  &= \sum_{X \in \Pcal} \sum_{Y \in N(X)} A_{n,\rho}(r(Y) / r(X))\\
  &\leq \dens_n(\rho) |\Pcal|.
  \end{split}
\]

Since~$\rho < 3$, we have~$A_{n,\rho}(1) > 0$. Putting it all together
we then get
\[
  \frac{2|E|}{|\Pcal|} \leq \frac{\dens_n(\rho)}{A_{n,\rho}(1)},
\]
finishing the proof.
\end{proof}

Note that~$A_{n,\rho}(1)$ is simply the normalized area of a spherical
cap of radius~$\alpha$ such that
\[
  \cos\alpha = \frac{\rho^2+3}{4\rho},
\]
and so we can compute~$A_{n,\rho}(1)$ explicitly for all~$n \geq
3$. By using in Theorem~\ref{thm:upper-bound} the trivial inequality
$\dens_n(\rho) \leq 1$ and taking~$\rho = \sqrt{3}$, we obtain upper
bounds for~$\avgkiss_n$: for~$n = 3$ we get the upper bound
of~$14.928\ldots$ from Kuperberg and Schramm~\cite{KuperbergS1994};
for~$n = 4$ and~$5$ we get the upper bounds~$34.680\ldots$
and~$77.756\ldots$ of Glazyrin~\cite{Glazyrin2017}. The
choice~$\rho = \sqrt{3}$ is optimal when using the upper
bound~$\dens_n(\rho) \leq 1$.

Extending techniques of Florian~\cite{Florian2001, Florian2007},
Glazyrin~\cite{Glazyrin2017} gives a better upper bound
for~$\dens_3(\sqrt{3})$, and so obtains~$\avgkiss_3 \leq 13.955$.


\section{Refining Glazyrin's approach using semidefinite programming}
\label{sec:first-sdp}

From Theorem~\ref{thm:upper-bound} we see that better upper bounds
for~$\dens_n(\rho)$ have the potential to give us better upper bounds
for the average kissing number. We will see now how semidefinite
programming can be used to provide upper bounds for~$\dens_n(\rho)$;
these upper bounds lead to improved upper bounds for the average
kissing number for~$n = 3$, \dots,~$9$ (see Table~\ref{tab:bounds}).

For fixed~$1 < \rho < 3$, the function~$A_{n,\rho}(r)$ is increasing
in~$r$ and has a limit at infinity, which we denote
by~$A_{n,\rho}(\infty)$. It is actually easy to compute this limit:
as~$r$ increases, the ball of radius~$r$ tangent to the central ball
of radius~$1$ resembles more and more a hyperplane tangent to the
central ball, so~$A_{n,\rho}(\infty)$ is the normalized area
of a spherical cap of radius~$\alpha$ such
that~$\cos\alpha = 1 / \rho$.

\begin{figure}[tb]
  \begin{center}
    \includegraphics{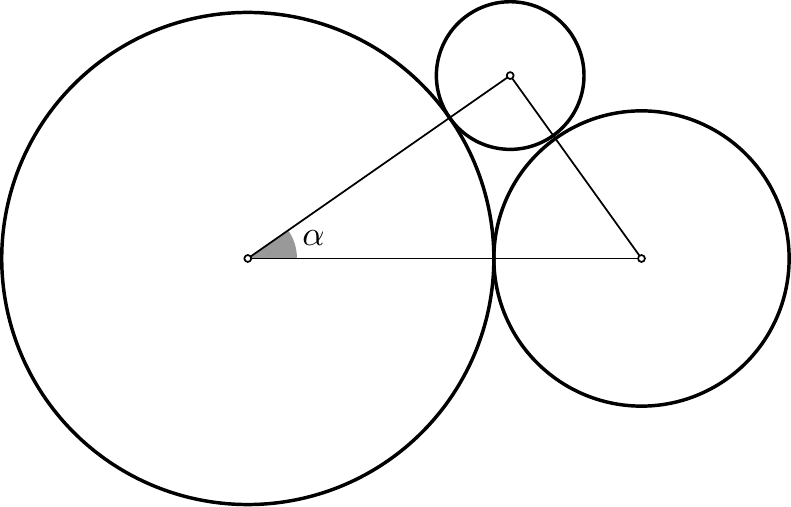}
  \end{center}

  \caption{Here, all circles are tangent. If the central circle has
    radius~$1$ and the other two have radii~$r$ and~$s$, then the law
    of cosines gives
    $(r+s)^2 = (1 + r)^2 + (1 + s)^2 - 2(1+r)(1+s)\cos\alpha$.}
  \label{fig:ip-def}
\end{figure}

Say two interior-disjoint balls of radii~$r$ and~$s$ touch a central
unit ball and let~$x$ and~$y$ be the contact points between each of
the balls and the central ball. Apply the law of cosines (see
Figure~\ref{fig:ip-def}) to get
\begin{equation}
  \label{eq:ip-def}
  x \cdot y \leq \frac{1 + r + s - rs}{1 + r + s + rs}.
\end{equation}
Denote the right-hand side above by~$\ip(r, s)$.

If~$F\colon [0,1]^2 \to \R$ is a kernel and~$U \subseteq [0, 1]$ is a
finite set, then the matrix $\bigl(F(u, v)\bigr)_{u, v \in U}$ is a
\defi{principal submatrix} of~$F$. We denote by~$P_k^n$ the Jacobi
polynomial of degree~$k$ and parameters~$\alpha = \beta = (n-3)/2$,
normalized so~$P_k^n(1) = 1$ (see the book by Szegö~\cite{Szego1975}
for background on Jacobi polynomials).

The following theorem is our basic tool to find upper bounds
for~$\dens_n(\rho)$.

\begin{theorem}
  \label{thm:dens-upper-bound}
  Let~$n \geq 3$ be an integer, $\rho$ be such that~$1 < \rho < 3$,
  and~$R$ be such that~$R > (\rho - 1) / 2$. Let $r$ be an increasing
  bijection from~$[0,1]$ to~$[(\rho-1)/2, R]$ and let
  $a\colon [0, 1] \to\nobreak \R$ be such that
  $a(u) \geq A_{n,\rho}(r(u))^{1/2}$ for all~$u \in [0, 1]$ and
  $a(1) \geq\nobreak A_{n,\rho}(\infty)^{1/2}$.

  Fix an integer~$d > 0$ and for every~$k = 0$, \dots,~$d$
  let~$F_k\colon [0,1]^2 \to \R$ be a kernel; write
  \begin{equation}
    \label{eq:f-def}
    f(t, u, v) = \sum_{k=0}^d F_k(u, v) P_k^n(t)\qquad\text{for~$t \in
      [-1, 1]$ and~$u$, $v \in [0,1]$.}
  \end{equation}
  If~$f$ and the kernels~$F_k$ are such that
  \begin{enumerate}
  \item[(i)] every principal submatrix of $F_0 - a \otimes a^*$ is
    positive semidefinite,

  \item[(ii)] every principal submatrix of $F_k$ is positive
    semidefinite for~$k = 0$, \dots,~$d$, and

  \item[(iii)] $f(t, u, v) \leq 0$ whenever~$-1 \leq t \leq \ip(r(u), r(v))$,
  \end{enumerate}
  then $\dens_n(\rho) \leq \max\{\, f(1, u, u) : u \in [0, 1]\,\}$.
\end{theorem}

This theorem is very similar to Theorem~1.2 of de Laat, Oliveira, and
Vallentin~\cite{LaatOV2014}. They consider configurations of spherical
caps of different radii, but the radii are taken from a finite list of
possibilities; then the function~$f$ is matrix valued. Here we work
with configurations of balls of different radii, and the list of
possible radii is infinite, namely the interval~$[(\rho - 1) / 2,
R]$. For this reason we work with the function~$f$ as defined in the
theorem above; $f$ can be seen as a kernel-valued function that
assigns to each~$t \in [-1, 1]$ a kernel on~$[0,1]^2$.

\begin{proof}
Consider any configuration~$\Pcal$ of interior-disjoint balls of any
radii tangent to the unit ball centered at the origin. Let~$\Delta$ be
the normalized area of~$\rho S^{n-1}$ covered by this configuration
and assume~$\Delta > 0$. Since a ball of radius less than~$(\rho-1) /
2$ tangent to the central unit ball does not intersect~$\rho S^{n-1}$,
we assume that each ball in~$\Pcal$ has radius at least~$(\rho - 1) /
2$.

Given a ball~$B \in \Pcal$, consider the point~$x \in S^{n-1}$ where
it touches the central unit ball. If the radius of~$B$ is in the
interval~$[(\rho-1) / 2, R]$, then let~$u \in [0, 1]$ be such
that~$r(u)$ is the radius of~$B$; otherwise, set~$u = 1$. Now~$B$ will
be represented by the pair~$(x, u)$;
let~$I \subseteq S^{n-1} \times [0, 1]$ be the set of pairs
representing each ball in~$\Pcal$.

We will need the following claim: if every principal submatrix
of~$F\colon [0,1]^2 \to \R$ is positive semidefinite and if~$k \geq 0$
is an integer, then the matrix
\[
  \bigl( F(u, v) P_k^n(x \cdot y) \bigr)_{(x, u), (y, v) \in I}
\]
is positive semidefinite.

The proof of this claim is as follows. Write
\[
  \begin{split}
    S &= \{\, x \in S^{n-1} : \text{$(x, u) \in I$ for some~$u \in
      [0,1]$}\,\}\quad\text{and}\\
    T &= \{\, u \in [0,1] : \text{$(x, u) \in I$ for some~$x \in
      S^{n-1}$}\,\}.
  \end{split}
\]
The addition theorem for Gegenbauer
polynomials~\cite[Theorem~9.6.3]{AndrewsAR1999} implies that there is
a real finite-dimensional Hilbert space~$H$ and vectors~$p(x) \in H$
for~$x \in S$ such that $P_k^n(x \cdot y) = \langle p(x), p(y)\rangle$
for all~$x$, $y \in S$, where~$\langle\cdot, \cdot\rangle$ denotes the
inner product in~$H$. Similarly, since every principal submatrix
of~$F$ is positive semidefinite, there is a real finite-dimensional
Hilbert space, which we may assume to be~$H$ as well, and
vectors~$q(u) \in H$ for~$u \in T$ such
that~$F(u, v) = \langle q(u), q(v)\rangle$ for all~$u$, $v \in T$. But
then
\[
  F(u, v) P_k^n(x \cdot y) = \langle p(x) \otimes q(u), p(y) \otimes
  q(v)\rangle
\]
for all~$(x, u)$, $(y, v) \in I$, and the claim follows.

Since~$P_0^n$ is the constant one polynomial, the claim just proved
together with (i), (ii), and the definition of~$f$ implies that the
matrix
\[
  \bigl(f(x\cdot y, u, v) - a(u) a(v)\bigr)_{(x, u), (y, v) \in I}
\]
is positive semidefinite. Hence
\[
  \sum_{(x, u), (y, v) \in I} f(x \cdot y, u, v) a(u) a(v) - a(u)^2
  a(v)^2 \geq 0.
\]

Since~$\Pcal$ is a configuration of interior-disjoint balls,
if~$(x, u)$, $(y, v) \in I$ then
$-1 \leq x \cdot y \leq \ip(r(u), r(v))$. Now use~(iii) and split the
sum on the left-hand side above into the diagonal and off-diagonal
terms to get the inequality
\[
  \begin{split}
    0 &\leq \sum_{(x, u), (y, v) \in I} f(x \cdot y, u, v) a(u) a(v)-
    a(u)^2 a(v)^2\\
    &\leq \sum_{(x, u) \in I} f(1, u, u) a(u)^2 -
    \Biggl(\sum_{(x, u) \in I} a(u)^2\Biggr)^2,
  \end{split}
\]
so
\[
  \sum_{(x, u) \in I} a(u)^2 \leq \max\{\, f(1, u, u) : u \in [0,
  1]\,\}.
\]
Finally, by the construction of~$I$ and the properties satisfied
by~$a$ we know that~$\Delta$ is at most the left-hand side above, and
so the theorem follows.
\end{proof}

To use Theorem~\ref{thm:dens-upper-bound} we need to specify the
kernels~$F_k$. One way to do so is to fix an integer~$N > 0$
and functions~$p_0$, \dots,~$p_N\colon [0,1] \to \R$. Then, given a
matrix~$A \in \R^{(N+1) \times (N+1)}$, set
\[
  F(u, v) = \sum_{i,j=0}^N A_{ij} p_i(u) p_j(v).
\]
It is easy to check that, if~$A$ is positive semidefinite, then every
principal submatrix of~$F$ is positive semidefinite. Similarly,
if~$a = \alpha_0 p_0 + \cdots + \alpha_N p_N$ and the matrix
$\bigl(A_{ij} - \alpha_i \alpha_j\bigr)_{i,j=0}^N$ is positive
semidefinite, then every principal submatrix of~$F - a \otimes a^*$ is
positive semidefinite.

In this way, Theorem~\ref{thm:dens-upper-bound} can be rephrased as a
semidefinite program. By choosing different
functions~$p_0$, \dots,~$p_N$, one obtains different optimization
problems, and there is an interplay between the functions chosen to
specify the kernels and the quality of the approximation~$a$
of~$u \mapsto A_{n,\rho}(r(u))^{1/2}$ that one can obtain. In the next
two sections we will use this approach to construct optimization
problems that give bounds for~$\dens_n(\rho)$ that lead to the new
upper bounds for the average kissing number in Table~\ref{tab:bounds};
for the functions~$p$ we will take alternately step functions and
polynomials.


\subsection{Step functions}

Let us first set the~$p_i$ to be step functions.
Fix~$R > (\rho - 1)/2$ and let~$r\colon [0, 1] \to [(\rho-1)/2, R]$ be
such that
\begin{equation}
  \label{eq:ru-step}
  r(u) = (R - (\rho - 1) / 2) u + (\rho-1)/2.
\end{equation}
Note that~$r$ is an increasing bijection between~$[0, 1]$
and~$[(\rho - 1) / 2, R]$.

Now fix an integer~$N > 0$ and
points~$0 = s_0 < s_1 < \cdots < s_N < s_{N+1} = 1$.
Let~$S_i = [s_i, s_{i+1})$ for~$i = 0$, \dots,~$N - 1$
and~$S_N = [s_N, s_{N+1}]$. Let~$p_i$ be the function that is~$1$
on~$S_i$ and~$0$ everywhere else.

The function~$a$ is now simple to specify: for $u \in [0, 1]$ set
\[
  a(u) = \begin{cases}
    A_{n,\rho}(r(s_{i+1}))^{1/2}&\text{if~$u \in S_i$ for some~$i <
      N$};\\
    A_{n,\rho}(\infty)^{1/2}&\text{if~$u \in S_N$.}
  \end{cases}
\]
Then~$a$ is an upper approximation of the
function~$u \mapsto A_{n,\rho}(r(u))^{1/2}$ (since this is a
monotonically increasing function) as needed in
Theorem~\ref{thm:dens-upper-bound}, and moreover~$a$ is a linear
combination of the~$p_i$ functions.

Each kernel~$F_k$ is parameterized by an~$(N+1) \times (N+1)$
positive-semidefinite matrix~$A_k$ as follows:
\[
  F_k(u, v) = \sum_{i,j=0}^N A_{k,ij} p_i(u) p_j(v)\qquad\text{for
    all~$u$, $v \in [0, 1]$},
\]
where~$A_{k,ij} = (A_k)_{ij}$.  So the kernels~$F_k$ are constant on
the sets~$S_i \times S_j$, and therefore~$F_k$ can be quite naturally
identified with~$A_k$. For fixed~$u$, $v \in [0, 1]$, the
function~$t\mapsto f(t, u, v)$ defined in~\eqref{eq:f-def} is a
polynomial on~$t$; for~$i$, $j = 0$, \dots,~$N$ we write~$f_{ij}(t)$
for the common value that~$f(t, u, v)$ assumes on~$S_i \times S_j$,
that is,
\[
  f_{ij}(t) = f(t, s_i, s_j) = \sum_{k=0}^d \sum_{i,j=0}^N A_{k,ij}
  p_i(s_i) p_j(s_j) P_k^n(t) = \sum_{k=0}^d A_{k,ij} P_k^n(t).
\]

Say that~$u \in S_i$ and~$v \in S_j$; from~\eqref{eq:ip-def} we
have~$\ip(r(u), r(v)) \leq \ip(r(s_i), r(s_j))$. So to ensure that~$f$
satisfies item~(iii) of Theorem~\ref{thm:dens-upper-bound} we have to
ensure that, for all~$i$, $j = 0$, \dots,~$N$,
\begin{equation}
  \label{eq:step-f}
  f_{ij}(t) \leq 0\qquad\text{whenever~$-1 \leq t \leq \ip(r(s_i),
    r(s_j))$}.
\end{equation}

Summarizing, if~$a = \alpha_0 p_0 + \cdots + \alpha_N p_N$, then any
feasible solution of the following optimization problem gives an upper
bound for~$\dens_n(\rho)$:
\begin{equation}
  \label{opt:dens-step}
  \optprob{\min&\multicolumn{2}{l}{\max\{\,f_{ii}(1) : \text{$i = 0$, \dots,~$N$}\,\}}\\
    &\multicolumn{2}{l}{f_{ij}(t) = \sum_{k=0}^d A_{k,ij} P_k^n(t),}\\[1ex]
    &f_{ij}(t) \leq 0&\text{whenever $-1 \leq t \leq \ip(r(s_i),
      r(s_j))$},\\
    &\bigl(A_{0,ij} - \alpha_i \alpha_j\bigr)_{i,j=0}^N&\text{is
      positive semidefinite,}\\[1ex]
    &A_k \in \R^{(N+1) \times (N+1)}&\text{is positive semidefinite for~$k =
      0$, \dots,~$d$.}
  }
\end{equation}

To model the nonpositivity constraints on the functions~$f_{ij}$ we
ensure nonpositivity on a finite sample of points
in~$[-1, \ip(r(s_i), r(s_j))]$. For the complete approach and a
description of how the solutions found by a solver can be rigorously
verified, see Appendix~\ref{ap:step-verify}.

Problem~\ref{opt:dens-step} can be used to give better bounds for the
average kissing number in dimensions~$5$, \dots,~$9$; see
Table~\ref{tab:bounds}. In dimension~$3$, we could not beat Glazyrin's
bound using this problem; in dimension~$4$, the bound provided is
better than Glazyrin's bound, but worse than the bound
of~\S\ref{sec:polynomials} below. To obtain the bounds of
Table~\ref{tab:bounds}, we used~$\rho = 2$, $N = 30$,
and~$R \approx 184.25$; see the verification script
(cf.~Appendix~\ref{ap:step-verify}) for precise information.


\subsection{Polynomials}
\label{sec:polynomials}

We now take the functions~$p_i$ to be polynomials.  Fix an
integer~$N > 0$ and let~$p_i(u) = u^i$ for~$i = 0$, \dots,~$N$.
Fix~$R > (\rho - 1) / 2$ and
let~$r\colon [0, 1] \to [(\rho - 1) / 2, R]$ be such that
\begin{equation}
  \label{eq:poly-ru}
  r(u) = (R - (\rho - 1) / 2) u^2 + (\rho - 1) / 2.
\end{equation}
Note that~$r$ is an increasing bijection between~$[0, 1]$
and~$[(\rho - 1) / 2, R]$. Note also that, in comparison
with~\eqref{eq:ru-step}, we have~$u^2$ instead of~$u$ above; we could
use~$u$, but we noticed that this leads to a worse upper
approximation~$a$.

To get the function~$a$ we solve a simple linear program
that is set up as follows. We have variables~$\alpha_0$,
\dots,~$\alpha_N$ for the coefficients of~$p_0$, \dots,~$p_N$. We fix
some~$\epsilon \geq 0$ and a finite sample~$S$ of points in~$[0, 1]$
and consider the constraints
\[
  \begin{array}{ll}
  \alpha_0 p_0(u) + \cdots + \alpha_N p_N(u) \geq
    A_{n,\rho}(r(u))^{1/2} + \epsilon&\text{for~$u \in S$},\\
    \alpha_0 p_0(1) + \cdots + \alpha_N p_N(1) \geq
    A_{n,\rho}(\infty)^{1/2}.
  \end{array}
\]
The objective is to minimize
\[
  \max\{\,\alpha_0 p_0(u) + \cdots + \alpha_N p_N(u) -
  A_{n,\rho}(r(u))^{1/2} : u \in S\,\}.
\]

\begin{figure}[tb]
  \begin{center}
    \includegraphics{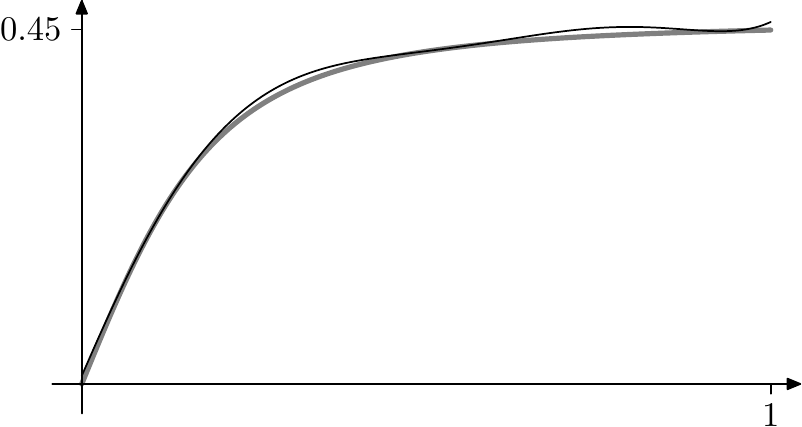}
  \end{center}

  \caption{A polynomial of degree~6 (in black) that
    approximates~$u \mapsto A_{3,\sqrt{3}}(r(u))^{1/2}$ (in gray) from
    above; here we take~$R =\nobreak 30$. Note the jump at the end to
    make~$a(1) \geq A_{3,\sqrt{3}}(\infty)$.}
  \label{fig:upper-approx}
\end{figure}

By taking a fine enough sample of points and a small
positive~$\epsilon \approx 10^{-5}$, the
function~$a = \alpha_0 p_0 + \cdots + \alpha_N p_N$ will satisfy the
properties in Theorem~\ref{thm:dens-upper-bound}, though this has to
be verified (see Appendix~\ref{ap:verify-poly}). If~$N$ is large
enough, the function~$a$ should also be close enough
to~$u \mapsto A_{n,\rho}(r(u))^{1/2}$ (see
Figure~\ref{fig:upper-approx}).

Each kernel~$F_k$ is parameterized by an~$(N+1) \times (N+1)$
positive-semidefinite matrix~$A_k$ as follows:
\[
  F_k(u, v) = \sum_{i,j=0}^N A_{k,ij} p_i(u) p_j(v)\qquad\text{for
    all~$u$, $v \in [0, 1]$}.
\]
So the function~$f$ defined in~\eqref{eq:f-def} is a polynomial
on~$t$, $u$, and~$v$. Recalling~\eqref{eq:ip-def}, item~(iii) in
Theorem~\ref{thm:dens-upper-bound} then asks that~$f(t, u, v)$ should
be nonpositive on the semialgebraic set~$\{\,(t, u, v) : \text{$s_i(t, u, v)
\geq 0$ for~$i = 1$, \dots,~$4$}\,\}$, where the~$s_i$ are the
following polynomials:
\begin{equation}
  \label{eq:sos-domain}
  \begin{split}
    s_1 &= t + 1,\\
    s_2 &= 1 + r(u) + r(v) - r(u) r(v) - t (1 + r(u) + r(v) + r(u) r(v)),\\
    s_3 &= u(1-u) + v(1-v),\quad\text{and}\\
    s_4 &= u(1-u) v(1-v);
  \end{split}
\end{equation}
note that~$s_3(t, u, v) \geq 0$ and~$s_4(t, u, v) \geq 0$ if and only
if~$0 \leq u \leq 1$ and~$0 \leq v \leq 1$.\footnote{Why note
  take~$s_3 = u(1-u)$ and~$s_4 = v(1-v)$? Then it is also true
  that~$s_3(t, u, v) \geq 0$ and~$s_4(t, u, v) \geq 0$ if and only
  if~$0 \leq u \leq 1$ and~$0 \leq v \leq 1$. The reason for our
  choice is that the polynomials~$s_3$ and~$s_4$
  in~\eqref{eq:sos-domain} are symmetric in~$u$ and~$v$, and this
  helps us reduce the size of the corresponding semidefinite program;
  see Appendix~\ref{ap:verify-poly-sol}.}

The nonpositivity condition can then be restricted to a sum-of-squares
condition: we require that there exist polynomials~$q_1$,
\dots,~$q_5$, each a sum of squares, such that
\begin{equation}
  \label{eq:sos-f}
  f = -s_1 q_1 - s_2 q_2 - s_3 q_3 - s_4 q_4 - q_5,
\end{equation}
since then~$f$ is clearly nonpositive in the required domain.

Finally, the upper bound on~$\dens_n(\rho)$ is given by the maximum
value of the function~$u \mapsto f(1, u, u)$ on~$[0, 1]$; note that
this function is a univariate polynomial on~$u$ of degree~$2N$. A
theorem of Lukács~\cite[Theorem~1.21.1]{Szego1975} says that this
maximum is equal to the minimum~$\lambda$ for which there are
univariate polynomials~$l_1$ and~$l_2$, each a sum of squares, such
that
\begin{equation}
  \label{eq:sos-obj}
  \lambda - f(1, u, u) = l_1(u) + u (1 - u) l_2(u).
\end{equation}

Putting it all together, we get the following optimization problem, any
feasible solution of which gives an upper bound for~$\dens_n(\rho)$:
\begin{equation}
  \label{opt:dens-sos}
  \optprob{\min&\lambda\\
    &f(t, u, v) = \sum_{k=0}^d \sum_{i,j=0}^N A_{k,ij} p_i(u) p_j(v)
    P_k^n(t),\\[1ex]
    &\text{$q_1$, \dots,~$q_5$ are sum-of-squares polynomials satisfying~\eqref{eq:sos-f},}\\[1ex]
    &\text{$l_1$ and $l_2$ are sum-of-squares polynomials
      satisfying~\eqref{eq:sos-obj},}\\[1ex]
    &\text{$\bigl(A_{0,ij} - \alpha_i \alpha_j\bigr)_{i,j=0}^N$ is
      positive semidefinite,}\\[1ex]
    &\text{$A_k \in \R^{(N+1) \times (N+1)}$ is positive semidefinite for~$k =
      0$, \dots,~$d$.}
  }
\end{equation}

Problem~\eqref{opt:dens-sos} can be rewritten as a
semidefinite program, where each sum-of-squares polynomial
is parameterized by a positive-semidefinite
matrix. Appendix~\ref{ap:verify-poly-sol} gives a detailed description
of the semidefinite program we solve and an overview of
how the solution found by the solver can be verified to be
feasible.

The approach of this section provides better bounds for the average
kissing number in dimensions~$3$ and~$4$; see
Table~\ref{tab:bounds}. In higher dimensions we could not manage to
obtain any bounds using this approach, since the problems are
infeasible when polynomials of low degree are used, and too large when
polynomials of high degree are used.


\section{A direct linear programming bound}
\label{sec:cohn-elkies}

Suppose we want to find the average kissing number, but that we
restrict ourselves to packings having balls of a few prescribed radii,
say~$r_1 < \cdots < r_N$. How many neighbors can a vertex in the
contact graph of such a packing have? Or, in other words, how many
balls can touch a given ball in the packing? Certainly, the largest
number of balls touching a central ball is attained when the central
ball has the largest possible radius,~$r_N$, and every ball touching
it has the smallest possible radius,~$r_1$. Hence the maximum degree
of the contact graph, and by consequence its average degree, is at
most the maximum number of interior-disjoint balls of radius~$r_1$
that can simultaneously touch a central ball of radius~$r_N$.

This is a simple upper bound for this restricted average kissing
number, but one could object it is too \textit{local}: the bound does
not take the whole packing into account, ignoring the interaction
between different balls. In particular, it is usually impossible for
every vertex in the packing to have the maximum possible degree, since
not every vertex can be the largest ball surrounded by several small
balls.

The bound for the average kissing number given by
Theorem~\ref{thm:upper-bound} appears to be similarly local. It is
based on the parameter~$\dens_n(\rho)$, which is not defined in terms
of a packing of balls and therefore cannot take into account the
interaction between different balls in a packing. We discuss now an
alternative idea, based on the linear programming bound of Cohn and
Elkies~\cite{CohnE2003} for the sphere packing density, which seems to
overcome this issue.

A continuous (matrix-valued)
function~$f\colon \R^n \to \R^{N \times N}$ is of \textit{positive
  type} if for every finite set~$U \subseteq \R^n$ the block matrix
\[
  \bigl(f(x-y)\bigr)_{x,y \in U}
\]
is positive semidefinite. Matrix-valued functions of positive type are
straightforward extensions of functions of positive type; see e.g.~the
paper by de Laat, Oliveira, and Vallentin~\cite[\S3]{LaatOV2014} for
more on such functions.

\begin{theorem}
  \label{thm:cohn-elkies}
  Let~$r_1$, \dots,~$r_N$ be any positive numbers. If~$f\colon \R^n
  \to \R^{N \times N}$ is a continuous function of positive type such
  that
  \begin{enumerate}
  \item[(i)] $f(x)_{ij} \leq 0$ if~$\|x\| \geq r_i + r_j$ and
  \item[(ii)] $f(x)_{ij} \leq -1$ if~$\|x\| = r_i + r_j$,
  \end{enumerate}
  then the average degree of the contact graph of a packing of balls
  of radii~$r_1$, \dots,~$r_N$ is at most $\max\{\, f(0)_{ii} :
  \text{$i = 1$, \dots,~$N$}\,\}$.
\end{theorem}

\begin{proof}
Let~$\Pcal$ be a packing of balls of radii~$r_1$, \dots,~$r_N$ and
let~$I \subseteq \{1, \ldots, N\} \times \R^n$ be such
that~$(i, x) \in I$ if and only if~$\Pcal$ has a ball of radius~$r_i$
centered at~$x$. Since~$f$ is of positive type we know that
\[
  \sum_{(i, x), (j, y) \in I} f(x-y)_{ij} \geq 0.
\]

Let~$G = (\Pcal, E)$ be the contact graph of the packing~$\Pcal$.
Split the sum above into three parts: the diagonal terms, the terms
corresponding to pairs of balls that do not touch, and the terms
corresponding to pairs of balls that do touch. Since~$f$ satisfies~(i)
and~(ii) we get
\[
  \begin{split}
    0 &\leq \sum_{(i, x) \in I} f(0)_{ii} + \sum_{\substack{(i, x),
        (j, y) \in I\\r_i+r_j \neq \|x-y\|}}  f(x-y)_{ij}
    + \sum_{\substack{(i, x),
        (j, y) \in I\\r_i+r_j = \|x-y\|}}  f(x-y)_{ij}\\
    &\leq \sum_{(i, x) \in I} f(0)_{ii} - 2|E|\\
    &\leq |I| \max\{\,f(0)_{ii} : \text{$i = 1$, \dots,~$N$}\,\} - 2|E|,
  \end{split}
\]
whence~$2|E| / |\Pcal| \leq \max\{\,f(0)_{ii} : \text{$i = 1$,
  \dots,~$N$}\,\}$, and the theorem follows.
\end{proof}

This theorem gives a direct bound for the average kissing number
instead of the rather indirect bound of Theorem~\ref{thm:upper-bound}
via the parameter~$\dens_n(\rho)$. Moreover, we have a
\textit{two-point bound}, that takes into account interactions between
pairs of balls in a packing.

Though Theorem~\ref{thm:cohn-elkies} is stated for a finite number of
possible radii, it can be easily extended to account for radii in any
bounded interval~$[a, b]$ if one uses kernel-valued functions. It is
not immediately clear, however, how to adapt the theorem for packings
of balls of arbitrary radii, and so Theorem~\ref{thm:cohn-elkies}
cannot be directly used to compute upper bounds for the average
kissing number.

For any given radii~$r_1$, \dots,~$r_N$, it is possible to reduce the
problem of finding functions satisfying the conditions in
Theorem~\ref{thm:cohn-elkies} to a semidefinite program;
such a reduction was employed by de Laat, Oliveira, and
Vallentin~\cite[\S5]{LaatOV2014} for a similar problem. In this way,
concrete bounds can be computed.

The bound of Theorem~\ref{thm:upper-bound} can be adapted to packings
of balls of finitely many possible radii~$r_1$, \dots,~$r_N$, namely
by changing the definition of~$\dens_n(\rho)$. Fix~$k = 1$,~\dots,~$N$
and~$\rho > 1$; consider the unit ball centered at the origin. Any
configuration of pairwise interior-disjoint balls of
radii~$r_i / r_k$, for~$i = 1$, \dots,~$N$, covers a certain fraction
of the sphere~$\rho S^{n-1}$ of radius~$\rho$ centered at the
origin. Denote the supremum of this covered fraction, taken over all
such configurations, by~$\dens_n^k(\rho)$, and let
\[
  \dens_n(\rho) = \max\{\dens_n^1(\rho), \ldots, \dens_n^N(\rho)\}.
\]
Following the proof of Theorem~\ref{thm:upper-bound}, we see that the
average degree of the contact graph of any packing of balls of
radii~$r_1$, \dots,~$r_N$ is at most
\[
  \frac{\dens_n(\rho)}{A_{n,\rho}(1)}.
\]
Upper bounds for~$\dens_n^k(\rho)$ can be computed using the approach
of de Laat, Oliveira, and Vallentin~\cite[Theorem~1.2]{LaatOV2014}.

So it is possible to compare numerically, for different choices of
radii, bounds given by Theorem~\ref{thm:cohn-elkies} with bounds given
by Theorem~\ref{thm:upper-bound}. The case~$N = 1$ is particularly
simple, since then~$\dens_n(\rho)$ is the kissing number~$\tau_n$
of~$\R^n$ times the area covered by the spherical cap, which
is~$A_{n,\rho}(1)$. So
\[
  \frac{\dens_n(\rho)}{A_{n,\rho}(1)} = \tau_n,
\]
and moreover any upper bound for the kissing number, like the
linear programming bound of Delsarte, Goethals, and
Seidel~\cite{DelsarteGS1977}, gives an upper bound for the average
degree of the contact graph. As for Theorem~\ref{thm:cohn-elkies}, we
have observed numerically that it provides a worse upper bound than
the linear programming bound for the kissing number. Surprisingly,
when more than one radius is considered, the bound of
Theorem~\ref{thm:cohn-elkies} becomes even worse;
Table~\ref{tab:cohn-elkies} contains some results.

\begin{table}[tbp]
  \begin{center}
  \begin{tabular}{ccc}
\textsl{Radii}&\textsl{Adapted
                Theorem~\ref{thm:upper-bound}}&\textsl{Theorem~\ref{thm:cohn-elkies}}\\[3pt]
    \hline\noalign{\vskip3pt}
    $1$              &13.159 & 13.402 \\
    $1/2$, $1$       &13.219 & 14.877 \\
    $1/3$, $1$       &13.159 & 17.294 \\
    $1/4$, $1$       &13.159 & 19.981 \\
    $1/5$, $1$       &13.159 & 22.770 \\
    $1/6$, $1$       &13.159 & 25.651 \\
    $1/3$, $1/2$, $1$&13.311 & 17.294 \\
    $1/4$, $1/2$, $1$&13.283 & 19.981 \\
    $1/4$, $1/3$, $1$&13.310 & 19.981 \\
    $1/5$, $1/3$, $1$&13.281 & 22.770 \\
    $1/4$, $1/3$, $1/2$, $1$&13.320& 19.981\\[3pt]
    \hline
  \end{tabular}
\end{center}
\bigskip

\caption{Comparison between bounds for the average degree of contact
  graphs of packings of balls of given radii in dimension~3 given by
  the adaptation of Theorem~\ref{thm:upper-bound} and
  Theorem~\ref{thm:cohn-elkies}. These bounds have been computed
  numerically.}
\label{tab:cohn-elkies}
\end{table}


\section*{Acknowledgements}

We would like to thank the Complex Systems and Big Data Competence
Centre at the University of Neuchâtel for providing access to their
computational cluster ``Cervino''.


%
%

\appendix

\section{Computing the area of a spherical cap}
\label{ap:area-caps}

To use Theorem~\ref{thm:dens-upper-bound} we need to be able to
compute the normalized area of a spherical cap of radius~$\alpha$
in~$S^{n-1}$, which is given by
\[
  \frac{\omega(S^{n-2})}{\omega(S^{n-1})} \int_{\cos\alpha}^1 (1 -
  u^2)^{(n-3)/2}\, du.
\]
The factor before the integral can be computed, to any desired
precision, by means of a recurrence. We will now derive a recurrence
relation for the integral above, making it possible to compute the
normalized area to any desired precision.

Let~$s$ be a real number. The Taylor series of~$u \mapsto (1 - u)^s$
around~$0$ is
\[
  \sum_{k=0}^\infty (-1)^k s (s-1) \cdots (s - k + 1) \frac{u^k}{k!}
  = \sum_{k=0}^\infty (-s)_k \frac{u^k}{k!},
\]
where for a real number~$a$ and an integer~$k \geq 0$ we denote
by~$(a)_k$ the \defi{shifted factorial}:
\[
  (a)_k = \begin{cases}
    1,&\text{if~$k = 0$;}\\
    a(a+1)\cdots(a+k-1),&\text{otherwise.}
  \end{cases}
\]

Substitute~$u$ by~$u^2$ and integrate term-by-term to get
\[
  \int (1 - u^2)^s\, du = \sum_{k=0}^\infty (-s)_k
  \frac{u^{2k+1}}{(2k+1) k!} = u\, {}_2F_1(1/2, -s; 3/2; u^2),
\]
where~${}_2F_1$ is the hypergeometric
series~\cite[Chapter~15]{AbramowitzS1964}.

So we want to compute
\[
  F_n(u) = {}_2F_1(1/2, -(n - 3) / 2; 3 / 2; u^2).
\]
Equation~(15.2.11) in the book by Abramovitz and
Stegun~\cite{AbramowitzS1964} gives us the relation
\begin{multline*}
  (c - b)\, {}_2F_1(a, b-1; c; z) + (2b - c - bz + az)\, {}_2F_1(a, b;
  c; z)\\ + b(z-1)\, {}_2F_1(a, b+1; c; z) = 0.
\end{multline*}
Take~$a = 1/2$, $b = -(n - 3) / 2$, and~$c = 3/2$ to get an
expression for~$F_{n+2}(u)$ in terms of~$F_n(u)$ and~$F_{n-2}(u)$. It
is now easy to obtain a recurrence; the base cases are
\begin{align*}
    F_3(u) &= 1,&
    F_4(u) &= \frac{u (1 - u^2)^{1/2} + \arcsin u}{2u},\\
    F_5(u) &= 1 - u^2 / 3,\qquad\text{and}&
    F_6(u) &= \frac{u (5 - 2u^2) (1 - u^2)^{1/2} + 3\arcsin u}{8u}.
\end{align*}


\section{The semidefinite program for step functions and rigorous
  verification}
\label{ap:step-verify}

For each~$i$, $j = 0$, \dots,~$N$, to implement the nonpositivity
constraint for~$f_{ij}$ in problem~\eqref{opt:dens-step} we select a
finite sample~$\Scal_{ij}$ of points in~$[-1, \ip(r(s_i),
r(s_j))]$. Consider the matrix~$W$ such
that~$W_{ij} = \alpha_i \alpha_j$ and fix~$\epsilon > 0$. To obtain an
upper bound we solve the following semidefinite program,
in which the role of the~$A_0$ variable changes in comparison
with~\eqref{opt:dens-step}:
\begin{equation}
  \label{opt:dens-step-sdp}
  \optprob{\min&\max\{\,f_{ii}(1) : \text{$i = 0$, \dots,~$N$}\,\}\\
    &f_{ij}(t) = W_{ij} + \sum_{k=0}^d A_{k,ij} P_k^n(t),\\
    &\text{$f_{ij}(t) \leq -\epsilon$\quad for~$t \in \Scal_{ij}$},\\
    &\text{$A_k \in \R^{(N+1) \times (N+1)}$ is positive semidefinite for~$k =
      0$, \dots,~$d$.}
  }
\end{equation}

In practice, we select samples of~50 points for each~$i$, $j$ and
set~$\epsilon = 10^{-5}$. We solve the resulting problem with standard
solvers and obtain a tentative optimal value~$z^*$.  The next step is
to remove the objective function and add it as a constraint, requiring
that
\[
  \max\{\,f_{ii}(1) : \text{$i = 0$, \dots,~$N$}\,\} \leq z^* +
  \eta,
\]
where~$\eta \approx 10^{-3}$. When we solve this feasibility problem,
the solver returns a strictly feasible solution, that is, a solution
in which every matrix~$A_k$ is positive definite. We observed that
this solution immediately satisfies the original nonpositivity
constraints of~\eqref{opt:dens-step}.

To verify that we have indeed a feasible solution, we only have to
verify that each~$A_k$ is positive semidefinite and compute an upper
bound on the value of~$f_{ij}$ on~$[-1, \ip(r(s_i), r(s_j))]$. Since
each~$A_k$ is actually positive definite, we use high-precision
floating-point arithmetic to compute for each~$A_k$ its Cholesky
decomposition~$L_k$. Then we replace~$A_k$ by~$L_k L_k^\tp$, so~$A_k$
becomes positive semidefinite by construction.

To get an upper bound for the value of~$f_{ij}$ on the corresponding
interval, we use interval arithmetic. We split the original interval
into subintervals and evaluate~$f_{ij}$ on each subinterval, obtaining
for each subinterval an upper bound on the value of~$f_{ij}$. In this
way, we obtain an upper bound~$u_{ij}$ on the value of~$f_{ij}$ on the
original interval.

Since the definition of~$f_{ij}$ uses~$W$, which in practice is
computed numerically, it is not enough to have~$u_{ij} \leq
0$. Indeed, if we use the exact value of~$W_{ij}$ instead of an
approximation, then~$f_{ij}$ could change to a positive number. To
prevent this from happening, we need to ensure that~$u_{ij}$ is
negative enough compared to the absolute error in the computation
of~$W_{ij}$. If we use the formulas of Appendix~\ref{ap:area-caps} to
compute~$W_{ij}$ using high-precision interval arithmetic, then we
have a rigorous bound on the absolute error of each~$W_{ij}$. This
whole verification approach is implemented in a Sage~\cite{SAGE}
script included with the arXiv version of this paper.


\section{The semidefinite program for polynomial interpolation and 
  rigorous verification}

Two steps are required in order to obtain rigorous upper bounds for
the average kissing number via the approach
of~\S\ref{sec:polynomials}. First, we must find a polynomial that
approximates the spherical-cap-area function from above, and we must
prove that this polynomial is really an upper bound for this
function. Second, we must rewrite problem~\eqref{opt:dens-sos} as a
semidefinite program, find good solutions for it, and prove that they
are feasible. In this section, we will see how both steps can be
carried out.


\subsection{Verifying the approximation for~$A_{n,\rho}$}
\label{ap:verify-poly}

In~\S\ref{sec:polynomials} we have seen how a polynomial~$a$
satisfying the conditions described in
Theorem~\ref{thm:dens-upper-bound} can be found. Here we quickly
describe how it can be rigorously verified that a given polynomial~$a$
satisfies these conditions; this verification approach is implemented
in a Sage~\cite{SAGE} script included with the arXiv version of this
paper.

Say~$\rho$ and~$R$ are fixed and let~$r(u)$ be defined as
in~\eqref{eq:poly-ru}. We want to prove
that~$a(u)^2 - A_{n,\rho}(r(u)) \geq 0$ for all~$u \in [0, 1]$ and
moreover that~$a(1)^2 - A_{n,\rho}(\infty) \geq 0$; the difficulty
lies in testing the validity of the first set of conditions.

Say we have an upper bound~$M$ on the absolute value of the derivative
of the function~$f(u) = a(u)^2 - A_{n,\rho}(r(u))$ on the
interval~$[0, 1]$. For an integer~$N \geq 1$, write~$\epsilon = 1 / (N
+ 1)$ and consider the points~$u_k = k\epsilon$ for~$k = 0$,
\dots,~$N+1$. If~$\eta$ is the minimum of~$f$ on the points~$u_k$,
then the mean-value theorem implies that
\[
  f(u) \geq \eta - \epsilon M / 2\qquad\text{for all~$u \in [0, 1]$.}
\]
So, as long as~$\eta - \epsilon M / 2 \geq 0$, the function~$f$ is
nonnegative on~$[0, 1]$. Computing the function~$f$ on lots of points
yields a guess for~$\eta$; then we find~$N$ such
that~$\epsilon = 1 / (N+1) \leq 2\eta/M$ and try to test the function
on the points~$u_k$. This is the approach implemented in the
verification script.

It remains to see how to compute the upper bound~$M$ on the absolute
value of the derivative of~$f$. Since~$a(u)$, and hence~$a(u)^2$, is a
polynomial, it is easy to compute an upper bound on the absolute value
of the derivative of~$a(u)^2$ rigorously using interval
arithmetic.

Bounding the derivative of~$A_{n,\rho}(r(u))$ is also simple. Indeed,
recall that~$A_{n,\rho}(s)$ equals
\[
  \frac{\omega(S^{n-2})}{\omega(S^{n-1})} \int_x^1 (1 -
  z^2)^{(n-3)/2}\, dz,
\]
where
\[
  x = \frac{\rho^2 + 2s + 1}{2\rho(1+s)}.
\]
So
\[
  A_{n,\rho}'(s) = -\frac{\omega(S^{n-2})}{\omega(S^{n-1})}
  (1-x^2)^{(n-3)/2} \frac{\rho^2-1}{2\rho(1+s)^2}.
\]
Note that~$0 \leq x \leq 1$, so $|(1-x^2)^{(n-3)/2}| \leq 1$. The
rightmost fraction above is largest when~$s = 0$ and
$|r'(u)| = |2 (R - (\rho-1)/2) u| \leq 2(R - (\rho-1)/2)$
for~$u \in [0,1]$. Finally, apply the chain rule to get
\[
  \biggl|\frac{d A_{n,\rho}(r(u))}{du}\biggr| \leq
  \frac{\omega(S^{n-2})}{\omega(S^{n-1})} \frac{|\rho^2-1|}{\rho} (R -
  (\rho-1)/2)
\]
for all~$u \in [0, 1]$.


\subsection{The semidefinite program and how to verify feasibility}
\label{ap:verify-poly-sol}

We quickly discuss how problem~\eqref{opt:dens-sos} is transformed
into a semidefinite program and then how a provably
feasible solution can be found for it.

To transform~\eqref{opt:dens-sos} into a semidefinite program, one has
to encode each sum-of-squares polynomial in terms of
positive-semidefinite matrices; here is the well-known
recipe. Let~$P_1$, \dots,~$P_m$ be a basis of the
space~$\R[x_1, \ldots, x_n]_{\leq k}$ of $n$-variable real polynomials
of degree at most~$k$ and for~$x = (x_1, \ldots, x_n)$ let~$v_k(x)$ be
the vector such that
\[
  (v_k(x))_i = P_i(x)\qquad\text{for~$i = 1$, \dots,~$m$.}
\]
We can see~$v_k$ as a ``vector'' whose entries are the
polynomials~$P_i$.

A polynomial~$p$ of degree~$2k$ is a sum of squares if and only
if there is a positive-semidefinite matrix~$X$ such that
\[
  p(x) = \langle v_k(x) v_k(x)^\tp, X\rangle,
\]
where~$\langle A, B\rangle = \trace AB$ is the trace inner product
between symmetric matrices~$A$ and~$B$. Note that above we have a
polynomial identity: both the left and right-hand sides are
polynomials that we require to be equal. This polynomial identity can
be rewritten as a set of linear constraints on the entries of the
matrix~$X$: for each monomial of degree at most~$2k$ we have one
constraint relating the coefficient of the monomial on both left and
right-hand sides.

The polynomials in~\eqref{eq:sos-obj} are univariate and have
degree~$2N$. So to model this constraint we
use~$v_N(u) = (1, u, u^2, \ldots, u^N)$.

Since the polynomial~$f(t, u, v)$ is symmetric on~$u$ and~$v$, that
is,~$f(t, u, v) = f(t, v, u)$, and since the same holds for the
polynomials~$s_1$, \dots,~$s_4$ in~\eqref{eq:sos-domain}, we are able
to further reduce the sizes of the positive-semidefinite matrices
needed to encode~\eqref{eq:sos-f}.

Indeed, say that a sum-of-squares polynomial~$p \in \R[t, u, v]$ is
symmetric on~$u$ and~$v$. It is not necessarily true that there is a
sum-of-squares decomposition~$p = q_1^2 + \cdots + q_m^2$ of~$p$ in
which every~$q_i$ is symmetric. However, we have that
\[
  \R[t, u, v] = \R[t, u+v, uv] \oplus (u - v) \R[t,
  u+v, uv],
\]
where~$\R[t, u+v, uv]$ is the ring of polynomials symmetric on~$u$
and~$v$. So say~$q_i = a_i + (u-v) b_i$, where~$a_i$ and~$b_i$ are
symmetric. Then
\[
  \begin{split}
    p(t, u, v) &= (1/2)(p(t, u, v) + p(t, v, u))\\
    &=\frac{1}{2} \sum_{i=1}^m a_i(t, u, v)^2 + 2(u-v)a_i(t, u,
    v)b_i(t, u, v) + (u - v)^2 b_i(t, u, v)^2\\
    &\phantom{=\frac{1}{2} \sum_{i=1}^m}\quad + a_i(t, v, u)^2 + 2(v-u)a_i(t, v,
    u)b_i(t, v, u) + (v - u)^2 b_i(t, v, u)^2\\
    &=\sum_{i=1}^m a_i(t, u, v)^2 + (u-v)^2 b_i(t, u, v)^2.
  \end{split}
\]
For any given~$k$, let~$v_k(t, u, v)$ be obtained from a basis
of~$\R[t, u+v, uv]_{\leq k}$ as before. The above discussion implies
that~$p \in \R[t, u, v]$ of degree~$2k$ is a symmetric sum of squares
if and only if there are positive-semidefinite matrices~$X$ and~$X'$
such that
\[
  p = \langle v_k v_k^\tp, X\rangle + \langle (u-v)^2 v_{k-1}
  v_{k-1}^\tp, X'\rangle.
\]

Now it should be clear how~\eqref{opt:dens-sos} can be rewritten as a
semidefinite program. Only one technical detail remains,
namely how to determine the degrees of the sum-of-squares polynomials
in~\eqref{opt:dens-sos}. Here we use for each polynomial the smallest
possible degree such that no term in the right-hand side
of~\eqref{eq:sos-f} has degree larger than~$d + 2N$, which is the
degree of~$f$.

The bounds for~$n = 3$ and~$4$ in Table~\ref{tab:bounds} were obtained
by solving the semidefinite program described above with~$d = 10$
and~$N = 6$, $8$, respectively. To solve the problem we use the
high-precision solver SDPA-GMP~\cite{Nakata2010}, but even so the
solution found by the solver is not truly feasible. To extract a
rational feasible solution from it we use the Julia library developed
by Dostert, de Laat, and Moustrou~\cite{DostertLM2020}. This allows us
to provide a rigorous upper bound for the average kissing number. The
script to generate the semidefinite program and to obtain an exact
rational solution is included with the arXiv version of this paper.


\begin{thebibliography}{17}
\bibitem{AbramowitzS1964}
M. Abramowitz and I.A. Stegun, {\it Handbook of mathematical functions
with formulas, graphs, and mathematical tables\/}, National Bureau
of Standards Applied Mathematics Series~55, U.S. Government Printing
Office, Washington, D.C., 1964.

\bibitem{AndrewsAR1999}
G.E. Andrews, R. Askey, and R. Roy, {\it Special Functions\/},
Encyclopedia of Mathematics and its Applications~71, Cambridge
University Press, Cambridge, 1999.

\bibitem{CohnE2003}
H. Cohn and N. Elkies, New upper bounds on sphere packings I, {\it
Annals of Mathematics\/}~157 (2003) 689--714.

\bibitem{ConwayS1988}
J.H. Conway and N.J.A. Sloane, {\it Sphere Packings, Lattices,
and Groups\/}, Grundlehren der mathematischen Wissenschaften~290,
Springer-Verlag, New York, 1988.

\bibitem{DelsarteGS1977}
P. Delsarte, J.M. Goethals, and J.J. Seidel, Spherical codes and
designs, {\it Geometriae Dedicata\/}~6 (1977) 363--388.

\bibitem{DostertLM2020}
M. Dostert, D. de Laat, and P. Moustrou, Exact semidefinite
programming bounds for packing problems, arXiv:2001.00256, 2020, 24pp.

\bibitem{EppsteinKZ2003}
D. Eppstein, G. Kuperberg, and G.M. Ziegler, Fat 4-polytopes and
fatter 3-spheres, in: {\it Discrete geometry\/}, Monographs and
Textbooks in Pure and Applied Mathematics~253, Dekker, New York, 2003,
pp.~239--265.

\bibitem{Florian2001}
A. Florian, Packing of incongruent circles on the sphere, {\it
Monatshefte für Mathematik\/}~133 (2001) 111--129.

\bibitem{Florian2007}
A. Florian, Remarks on my paper: ``Packing of incongruent circles on
the sphere'', {\it Monatshefte für Mathematik\/}~152 (2007) 39--43.

\bibitem{Glazyrin2017}
A. Glazyrin, Contact graphs of ball packings, arXiv:1707.02526, 2017,
15pp.

\bibitem{KuperbergS1994}
G. Kuperberg and O. Schramm, Average kissing numbers for non-congruent
sphere packings, {\it Mathematical Research Letters\/}~1 (1994)
339--344.

\bibitem{LaatOV2014}
D. de Laat, F.M. de Oliveira Filho, and F. Vallentin, Upper bounds
for packings of spheres of several radii, {\it Forum of Mathematics,
Sigma\/}~2 (2014) e23.

\bibitem{MachadoO2018}
F.C. Machado and F.M. de Oliveira Filho, Improving the semidefinite
programming bound for the kissing number by exploiting polynomial
symmetry, {\it Experimental Mathematics\/}~27 (2018) 362--369.

\bibitem{Nakata2010}
M. Nakata, A numerical evaluation of highly accurate
multiple-precision arithmetic version of semidefinite programming
solver: SDPA-GMP,-QD and-DD, in: {\it 2010 IEEE International
Symposium on Computer-Aided Control System Design\/}, 2010,
pp.~29--34.

\bibitem{SAGE}
W.A. Stein {\it et al.}, Sage Mathematics Software (Version 6.3), The
Sage Development Team, 2014, {\tt http://www.sagemath.org}.

\bibitem{Stephenson2005}
K. Stephenson, {\it Introduction to Circle Packing: The Theory of
Discrete Analytic Functions\/}, Cambridge University Press, Cambridge,
2005.

\bibitem{Szego1975}
G. Szegö, {\it Orthogonal Polynomials\/} (Fourth Edition), American
Mathematical Society Colloquium Publications Volume~XXIII, American
Mathematical Society, Providence, 1975.
\end{thebibliography}
\end{document}